\documentclass[12pt,a4paper]{article}
\usepackage{amssymb,amsmath,amsthm,enumerate}
\usepackage{graphicx}

\newcommand\cM{{\mathcal M}}

\theoremstyle{plain}
\newtheorem{theorem}{Theorem}[section]
\newtheorem{lemma}[theorem]{Lemma}

\newtheorem{conjecture}[theorem]{Conjecture}
\newtheorem{proposition}[theorem]{Proposition}

\theoremstyle{definition}

\newtheorem{defn}[theorem]{Definition}

\newtheorem{claim}[theorem]{Claim}

\newtheorem*{remark}{Remark}

\newcommand\cref[1]{Corollary~\ref{cor:#1}}

\title{On the maximum number of copies of $H$ in graphs with given size and order}
\author{D\'aniel Gerbner%\thanks{Research supported by the J\'anos Bolyai Research Fellowship of the Hungarian Academy of Sciences and the National Research, Development and Innovation Office -- NKFIH under the grant K 116769.} 
, D\'aniel T. Nagy%\thanks{Research supported by the National Research, Development and Innovation Office - NKFIH under the grant K 116769.}
, Bal\'azs Patk\'os%\thanks{Research supported by the National Research, Development and Innovation Office -- NKFIH under the grants SNN 116095 and K 116769.}
, M\'at\'e Vizer \\ %\thanks{Research supported by the National Research, Development and Innovation Office -- NKFIH under the grant SNN 116095.}
\medskip 
\small Alfr\'ed R\'enyi Institute of Mathematics, Hungarian Academy of Sciences\\
\small P.O.B. 127, Budapest H-1364, Hungary.\\
\medskip
\small \texttt{\{gerbner,nagydani,patkos\}@renyi.hu, vizermate@gmail.com}
\medskip}

\begin{document}
\maketitle
\begin{abstract} We study the maximum number $ex(n,e,H)$ of copies of a graph $H$ in graphs with given number of vertices and edges. We show that for any fixed graph $H$, $ex(n,e,H)$ is asymptotically realized by the quasi-clique provided that the edge density is sufficiently large. We also investigate a variant of this problem, when the host graph is bipartite.

\end{abstract}

\section{Introduction}
In this paper, we address (some variants of) the problem of finding the graph that maximizes the number of copies of some fixed graph $H$ over all $n$-vertex graphs with $e$ edges. Formally, we denote by $N(H,G)$ the number of copies of $H$ in $G$ and we will be interested in determining $ex(n,e,H)=\max\{N(H,G): |V(G)|=n,|E(G)|=e\}$. 
%(We can note here somewhere that )
In the early 1970's Ahlswede and Katona proved \cite{ak} that the number of cherries (stars with two leaves or equivalently paths on three vertices) is always maximized by one of the following two graphs: 
\begin{itemize}
\item
if $e=\binom{a}{2}+b$ with $0\le b<a$, then the \textit{quasi-clique} $K^e_n$ contains a clique of size $a$ and at least $n-a-1$ isolated vertices,
\item
the \textit{quasi-star} $S^e_n$ is the complement of $K^{\binom{n}{2}-e}_n$.
\end{itemize}

Generalizing cherries, Kenyon, Radin, Ren, and Sadun \cite{krrs} determined the asymptotic value of $ex(n,e,S_j)$ for small values of $j$ with $S_j$ being the star with $j$ leaves. Their result was extended to arbitrary stars by Reiher and Wagner.

\begin{theorem}[Reiher, Wagner \cite{Reiher}]\label{reiher}
For any $j\ge 2$ we have $$ex(n,e,S_j)=\max(N(S_j,K^e_n),N(S_j,S^e_n))+O(n^{j}).$$
For any $j\ge 2$ there exists a number $0<\gamma_j<1$ such that the asymptotically best construction is $S^e_n$ if $0\le e\le \gamma_j\binom{n}{2}$ and $K^e_n$ if $\gamma_j\binom{n}{2}\le e\le \binom{n}{2}$.
\end{theorem}

Nagy \cite{Dani} determined the asymptotic value of $ex(n,e,P_4)$. Answering a question of his (Question 3 in \cite{Dani}) in the affirmative, we prove the following theorem.

\begin{theorem}\label{kvaziklikk}

For every $H$ there exists a constant $0<c_H < 1$ such that if $e\ge c_H\binom{n}{2}$, then among graphs with $n$ vertices and $e$ edges, the number of subgraphs isomorphic to $H$ is asymptotically maximized by the quasi-clique, i.e. $ex(n,e,H)=(1+o(1))N(H,K^e_n)$.

\end{theorem}

Let us mention that results on the maximizing $N(H,G)$ over all graphs containing $e$ edges with an arbitrary number of vertices were obtained by Alon \cite{alon1,alon2}, Bollob\'as and Sarkar \cite{bs,bs2}, and F\"uredi \cite{F}. Bollob\'as, Nara, and Tachibana \cite{BNT} studied the maximum number of \emph{induced} copies of $K_{s,t}$ in graphs with $e$ edges.

%$\bullet$ $ex(e,H)$

\subsection{Bipartite host graphs}

\ 

We will consider similar problems for bipartite host graphs. Let $ex_{bip}(n,e,H)$ denote the maximum possible number of copies of $H$ in a bipartite graph $B$ on $n$ vertices with $e$ edges. Formally, let $$ex_{bip}(n,e,H):=\max\{N(H,B): |V(B)|=n,E(B)=e, B ~\text{is bipartite} \}.$$ 

Note that in this variant of the problem there is no monotonicity in $e$; for a given $n$, if $e>e'$, then it is still possible that $ex_{bip}(n,e,H)<ex_{bip}(n,e',H)$, even when $e$ is less than the maximum number $\lfloor n^2/4\rfloor$ of edges in a bipartite graph on $n$ vertices, because more edges necessarily make the two parts more balanced. We will show an example at the end of Section 2.

Observe that if $H$ is a complete bipartite graph $K_{s,t}$, then every copy of $H$ in $B$ is an induced copy. Thus $ex_{bip}(n,e,H)$ is upper bounded by the maximum number of induced copies of $K_{s,t}$ in graphs with $e$ edges, provided that this maximum is given by a bipartite graph. Bollob\'as, Nara, and Tachibana \cite{BNT} showed that this is the case for $H=K_{t,t}$ in case $e=r^2$ or $e=r^2-1$ or $e=r^2+r$ for some $r$. In these cases the maximum number of induced copies of $K_{t,t}$ is given by $K_{r,r}$, $K_{r-1,r+1}$ and $K_{r,r+1}$ respectively, which implies $ex_{bip}(n,r^2,K_{t,t})=N(K_{t,t},K_{r,r})$, $ex_{bip}(n,r^2-1,K_{t,t})=N(K_{t,t},K_{r-1,r+1})$ and $ex_{bip}(n,r^2+r,K_{t,t})=N(K_{t,t},K_{r+1,r})$ respectively.

The bipartite analogue of the quasi-clique is the \textit{quasi-complete bipartite graph} $B_n^e$ defined below.

\begin{defn}
Let $n$ and $e$ be two positive integers satisfying $e\le n^2/4$. Let $t$ be the smallest positive integer $i$ such that $i(n-i)\ge e$. (Then $1\le t\le n/2$.) The graph $B_n^e$ is obtained from the complete bipartite graph $K_{t,n-t}$ by picking a vertex of degree $n-t$ and deleting $t(n-t)-e$ of its edges.
\end{defn}

Note that by the definition of $t$, we have $(t-1)(n-t+1)\le e-1$, hence the vertex that we picked will have (remaining) degree $(n-t)-(t(n-t)-e)=e-(t-1)(n-t)\ge t$.

%$B^e_n$. If $a(n-a)= e$, then $B^e_n$ is the complete bipartite graph $K_{a,n-a}$, while if $a(n-a)<e< (a+1)(n-a-1)$, then $B^e_n$ consists of two parts $U_1,U_2$ with $|U_1|=a+1,|U_2|=n-a-1$ all but one vertex $u\in U_1$ is connected to all vertices of $U_2$ and this exceptional vertex $u$ is connected to $e-a(n-a-1)$ vertices of $U_2$.
% \ 

% \textbf{Asymptotic result}

\begin{conjecture}\label{paros} For any integers $e,n$ with $\omega(n)=e\le \frac{n^2}{4}$ and any bipartite graph $H$ we have
$$ex_{bip}(n,e,H)=(1+o(1))N(H,B^e_n).$$ 

\end{conjecture}
Note that this does not contradict the exact results mentioned earlier, as $N(K_{t,t},K_{r,r})=(1+o(1))N(K_{t,t},B^{r^2}_n)$ if $\omega(n)=e$.
Observe furthermore that the assumption $\omega(n)=e$ is necessary as $K_{a,n-a}$ does not contain any copy of $K_{a+1,a+1}$. We prove Conjecture \ref{paros} in a stronger exact form for stars.

\begin{theorem}\label{paroscsillag}
For any integers $e,n$ with $e\le \frac{n^2}{4}$ we have
$ex_{bip}(n,e,S_k)=N(S_k,B^e_n)$.

\end{theorem}

Using Theorem \ref{paroscsillag} we can verify conjecture \ref{paros} for a wider class of graphs.

\begin{theorem}\label{parosasy}
If the vertex set of a bipartite graph $B$ is covered by vertex disjoint edges and possibly a star $S_j$ $(j\ge 2)$, then we have $$ex_{bip}(n,e,B)=(1+o(1))N(B,B^e_n).$$
\end{theorem}

Note that Theorem \ref{parosasy} covers the case $H=K_{s,t}$ for any fixed pair $s,t$ of integers.

\

% \textbf{Exact results}

% $\bullet$ Let $OPT(e,s,t)$ and $OPT(n,e,s,t,)$ be the solutions of the following quadratic optimization problems:

% 1. ...

% 2. ...

% \begin{theorem}\label{exactbip}

% $ex_{bip}(n,e,K_{s,t})=OPT(n,e,s,t)$

% \end{theorem}

% Possibly $K_{s,t}$ instead of $C_4$

\section{Proof of Theorem \ref{paroscsillag}}
If $G$ is a graph with vertex degrees $d_1, d_2,\dots, d_n$ then the number of $k$-edge stars in $G$ is
$$N(S_k,G)=\sum_{i=1}^n \binom{d_i}{k}.$$

We say that a bipartite graph is \textit{neighbor-nested}, if for any two vertices $x$ and $y$ from the same vertex class we have $N(x)\subseteq N(y)$ or $N(y)\subseteq N(x)$. We claim that among the bipartite graphs with $n$ vertices and
$e$ edges, the maximum number of $k$-edge stars can be achieved by a neighbor-nested graph. To prove this, assume that a bipartite graph $G$ is not neighbor-nested. Then there are two vertices $x$ and $y$ on the same side such that $\deg(x)\le \deg(y)$ but there is a vertex $z\in N(x)\backslash N(y)$. Let us replace the edge $xz$ by $yz$. The change in the number of $k$-edge stars is
$$\binom{\deg(y)+1}{k}-\binom{\deg(y)}{k}-\binom{\deg(x)}{k}+\binom{\deg(x)-1}{k}=
\binom{\deg(y)}{k-1}-\binom{\deg(x)-1}{k-1}\ge 0.$$
So the number of $k$-edge stars increases or stays the same during such a transformation. By the above calculation, the number of 2-edge stars always increases by at least 1. This quantity is obviously bounded by $\binom{e}{2}$, so after finitely many steps, the process must stop. This means that we reached a neighbor-nested graph that has at least as many $k$-edge stars as our original graph.

Assume that $G$ is a neighbor-nested bipartite graph, where the degrees of the vertices on one side are $a_1\ge a_2\ge\dots \ge a_m$ and on the other side they are $b_1\ge b_2\ge\dots \ge b_t$. We will visualize $G$ as follows. Consider the upper right quarter of the plane, where unit squares represent the edges of $G$. The first column has $a_1$ squares in it, the next one has $a_2$ and so on. Similarly, the bottom row consists of $b_1$ squares, the ones above it having $b_2,\dots, b_t$.

\begin{figure}[h]
\begin{center}
\includegraphics[scale = 0.5] {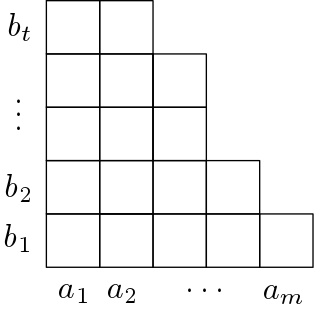}
\end{center}
\caption{A neighbor-nested graph visualized}
\label{fig1}
\end{figure}

We define the weight of a square in the $i$-th column and $j$-th row as $$w(i,j)=\binom{i}{k}-\binom{i-1}{k}+\binom{j}{k}-\binom{j-1}{k}=\binom{i-1}{k-1}+\binom{j-1}{k-1}.$$
Then the total weight of the squares is equal to
$$\sum_{i=1}^m \binom{a_i}{k}+\sum_{j=1}^t \binom{b_i}{k}=N(S_k,G).$$

It follows from the definition of the weight function that is has the following properties:
\begin{enumerate}[i)]
\item $w(i,j)=w(j,i)$ for all integers $i,j\ge 1$.
\item $w(i_1,j)\le w(i_2,j)$ for all integers $1\le i_1<i_2$ and $1\le j$.
\item $w(x-z,y)\le w(x,y-z)$ for all integers $1\le z < y\le x$.
\end{enumerate}

From now on we will only need to use these simple properties of the weight function, instead its exact definition.

Now we define an algorithm to transform $G$ while not decreasing the total weight of the squares. This algorithm will eventually reach $B^e_n$, showing its maximality.

{\bf Step 1:} Find the largest square $S$ whose bottom left corner is 0 in the diagram. If there are unit squares above $S$, we move them to the right side of $S$. Precisely, the unit squares in the $i$-th column above $S$ are added to the end of the $i$-th row.

\begin{figure}[h]
\begin{center}
\includegraphics[scale = 0.33] {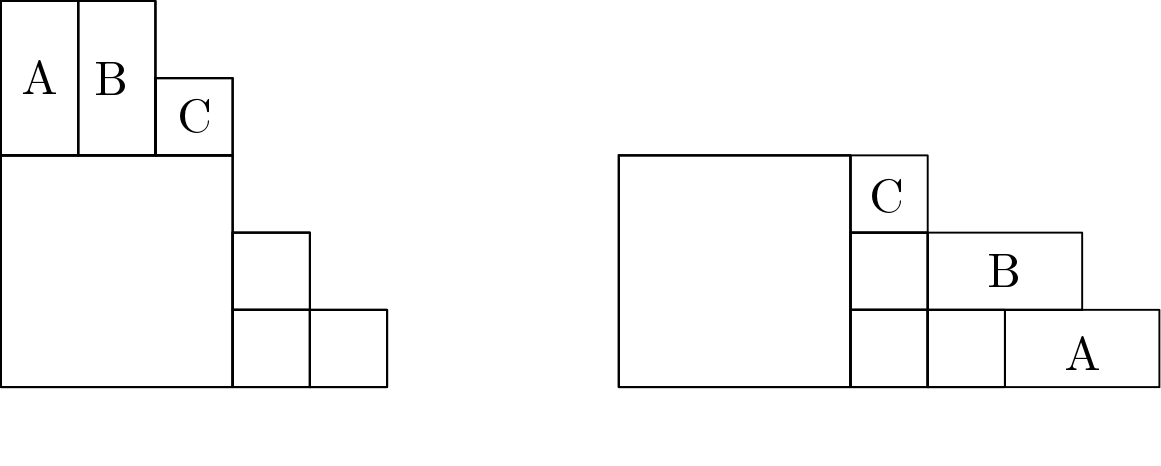}
\end{center}
\caption{An example for Step 1}
\label{fig2}
\end{figure}

By properties i) and ii) this transformation does not decrease the total weight. Also, the total number of vertices used on the two side of $G$ stays the same. The resulting graph will be neighbor-nested as well. After this transformation is completed (or there is nothing to do), we move the Step 2.

Note that after Step 1, there will be at most $a_1$ unit squares in any column, but there will be at least $a_1$ in any row.

{\bf Step 2:} Consider the second row from the top (it has $b_{t-1}$ squares) and the empty square right after it. We want to fill this and any other empty place directly under it using squares removed from the end of the top row. Since the top row has at least $a_1$ squares and we want to fill at most $a_1-1$ places, we have enough squares to work with. The transformation preserves the neighbor-nested property.

\begin{figure}[h]
\begin{center}
\includegraphics[scale = 0.7] {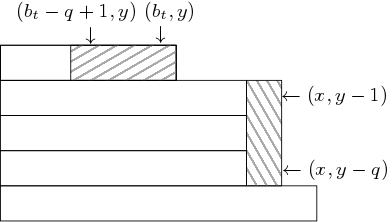}
\end{center}
\caption{An example for Step 2}
\label{fig3}
\end{figure}

Let us use the notations $x=b_{t-1}+1,~y=a_1$ and $q$ for the number of unit squares moved during the step. Then the weight of the moved squares will change from $w(b_t-q+1,y)+ \dots +w(b_t,y)$ to $w(x,y-q)+ \dots +w(x,y-1)$. Using property ii) then iii), we conclude that
$$w(b_t-q+1,y)+ \dots +w(b_t,y)\le w(x-q,y)+ \dots +w(x-1,y) \le w(x,y-q)+ \dots +w(x,y-1).$$

There are only two reasons that would stop us from doing this step. 1) There is only one row in the whole diagram. Then $G$ is of the form $B^e_n$. (This will happen if $e<n$.)
2) All the rows below the second-from-top have the same length as it and $m+t=n$. Then we are already using the maximum allowed number of vertices, and the step would create a new one. This means that $G=B^e_n$, and we are done.

If we completed Step 2, and did not arrive at $B^e_n$, then continue with Step 1 again.

We showed that the above algorithm does not decrease the number of $k$-edge stars, and it can only stop at $B_n^e$. To show that it indeed reaches $B_n^e$ (instead of entering an infinite loop) consider the vector $(a_1,a_2,\dots a_m)$ after each step. Observe that with respect to the lexicographic order, this vector decreases or stays intact during Step 1 and strictly decreases during Step 2. Therefore the algorithm must stop after finitely many transformations. 

\begin{remark}
%Somewhat surprisingly, i
It is possible that $N(S_k,B_n^e)>N(S_k,B_n^{e+1})$. It means that in some cases even if we have more edges to work with, we can create only fewer $k$-stars. To see an example, let us count 2-edge stars in $B_8^{12}$ and $B_8^{13}$.

\begin{figure}[h]
\begin{center}
\includegraphics[scale = 0.4] {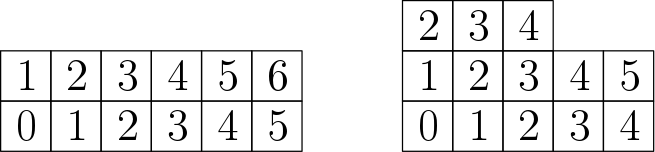}
\end{center}
\caption{The diagrams of $B_8^{12}$ and $B_8^{13}$ with the values of the weight function}
\label{fig4}
\end{figure}

$N(S_k,B_8^{12})=36$ and $N(S_k,B_8^{13})=34$. This situation can occur because one cannot create a graph with 13 edges by adding a new edge to the optimal 12-edge construction since it is already a complete bipartite graph using all the allowed vertices.
\end{remark}
%\section{Proof of Theorem \ref{exactbip}}

\section{Proof of the asymptotic results}

In this section we prove Theorem \ref{kvaziklikk} and Theorem \ref{parosasy}. The proof of these results follow the same lines: we prove or assume the existence of a spanning forest of $H$ the components of which satisfy the statement of the theorems, then we show by a simple computation that the number of these spanning forests is asymptotically maximized by the graphs stated in the theorems. Finally, we show that every such spanning forest can be extended only in one way to a copy of $H$ and in the extremal graphs almost every such forest indeed extends to a copy of $H$.

\begin{lemma}\label{count}
(i) For integers $e,n$ with $e\le \binom{n}{2}$ let $v_e$ denote $\max\{v:\binom{v}{2}\le e\}$. Let $a_1,a_2,\dots,a_k$ be a fixed sequence of positive integers and let $e=e_n$ tend to infinity. Then the number $S_{vd}(a_1,a_2,\dots,a_k,e)$ of vertex disjoint $k$-tuples $(S_1,S_2,\dots,S_k)$ of stars with $S_i$ having $a_i$ leaves in $K^e_n$ is $(1-o(1))\prod_{i=1}^k(a_i+1)\binom{\sqrt{v_e+1}}{a_i+1}$.

(ii) Let $k\ge 2$, $m\ge 0$ be fixed integers and let $e=e_n=\omega(n)$. Define $a=a_{e,n}=\max\{\alpha\le n/2: \alpha(n-\alpha)\le e\}$. Then the number of pairs $(S_k,\cM)$, where $S_k$ and $\cM$ are vertex-disjoint subgraphs of $B^e_n$ with $S_k$ being a star with $k$ leaves and $\cM$ being a matching of $m$ edges is $(1-o(1))\binom{e}{m}((a+1)\binom{n-a}{k}+(n-a)\binom{a+1}{k})$.
\end{lemma}

\begin{proof}
Observe that the number of non-isolated vertices in $K_n^e$ is at most $v_e+1$, so we have $S_{vd}(a_1,a_2,\dots,a_k)\le \prod_{i=1}^k(a_i+1)\binom{v_e+1}{a_i+1}$. On the other hand, the size of the largest clique in $K_n^e$ is $v_e$, therefore, if $s=\sum_{i=1}^k(a_i+1)$, then $S_{vd}(a_1,a_2,\dots,a_k)\ge \prod_{i=1}^k(a_i+1)\binom{v_e-s}{a_i+1}$. As $e$ tends to infinity, so does $v_e$ and therefore $\binom{v_e-s}{a_i+1}/\binom{v_e+1}{a_i+1}$ tends to 1 for all $i=1,2,\dots,k$. This proves (i).

The expression $\binom{e}{m}((a+1)\binom{n-a}{k}+(n-a)\binom{a+1}{k})$ is clearly an upper bound in (ii). The lower bound follows from the fact that $K_{a,n-a-1}$ is a subgraph of $B^e_n$ and that for edges $e_1,e_2$ and a $k$-star $S_k$ of $K_{a,n-1}$ picked uniformly at random we have 
\begin{itemize}
\item
$\mathbb{P}(e_1\cap e_2\neq \emptyset)\le \frac{a\binom{n-a}{2}+(n-a-1)\binom{a}{2}}{\binom{a(n-a-1)}{2}}\rightarrow 0$ as $e=\omega(n)$ implies $a\rightarrow \infty$,
\item
$\mathbb{P}(e_1\cap S_k\neq \emptyset)\le \frac{a\binom{n-a-1}{k}(n-a+ka)+(n-a-1)\binom{a}{k}(a+k(n-a))}{e[a\binom{n-a-1}{k}+(n-a-1)\binom{a}{k}]} \rightarrow 0$ as $e=\omega(n)$ implies $a\rightarrow \infty$.
\end{itemize}
\end{proof}

%sketch: we take a spanning forest, if it is maximized, than the number of ways to extend it to $H$ is maximal. For the forest we take the partition from the next lemma, so it is enough to have the result for stars, Reiher and Wagner

\begin{proposition}\label{propo}
The vertex set $V$ of every tree $T$ on at least two vertices can be partitioned into $U_1,U_2,\dots,U_k$ such that $T[U_i]$ is a star with at least two vertices for all $i=1,2,\dots,k$.
\end{proposition}

\begin{proof}
We proceed by induction on $|V|$. If $|V|\le 4$, the statement follows as all trees on at most 4 vertices are either stars or paths (or both). If $|V|\ge 5$, consider a non-leaf vertex $v\in T$ and its neighbors $u_1,u_2,\dots,u_j$. If all $u_i$'s are leaves, then $T$ is a star, so we are done by using the partition with one single part. If, say, $u_1$ has a different neighbor than $v$, then we apply induction to $T_1$, the subtree of all vertices that are closer to $u_1$ than to $v$, and to $T_2=T\setminus T_1$.
\end{proof}

\begin{proof}[Proof of Theorem \ref{kvaziklikk}]
For any pair $H,G$ of graphs let $h(H,G)$ denote the number of injective homomorphisms from $H$ to $G$. As $N(H,G)=\frac{1}{a}h(H,G)$ where $a$ is the number of automorphisms of $H$, it is enough to prove the statement for $h(H,G)$. Let $T$ be a spanning tree of $H$ and let $U_1,U_2,\dots, U_k$ be a partition of the vertex set of $T$ into stars given by Proposition \ref{propo}. Clearly, we have $$h(H,G)\le h(T[U_1]+T[U_2]+\dots+T[U_k],G)\le \prod_{i=1}^kh(T[U_i],G).$$
Let us define $a_i$ by $T[U_i]=S_{a_i}$ for $i=1,\dots, k$. Recalling the definition of the constants $\gamma_j$ from Theorem \ref{reiher} we define $c_H:=\max\{\gamma_{a_i},~i=1,\dots, k\}$. If $e(G)\ge c_H\binom{n}{2}$ then by Theorem \ref{reiher} the right hand side is at most $(1+o(1))\prod_{i=1}^kh(S_{a_i},K^{e(G)}_n)=(1+o(1))\prod_{i=1}^k(a_i+1)\binom{v_e+1}{a_i+1}$

%By Theorem \ref{reiher} the right hand side is at most $(1+o(1))\prod_{i=1}^kh(T[U_i],K^{e(G)}_n)$ if $e(G)$ is large enough and this is $(1+o(1))\prod_{i=1}^k(a_i+1)\binom{v_e+1}{a_i+1}$ where $a_i$ is defined by $T[U_i]=S_{a_i}$.
Now observe that by Lemma \ref{count} (i) the number $h(T[U_1]+T[U_2]+\dots+T[U_k],K^{\binom{v_{e(G)}}{2}}_n)$ is asymptotically the same as the above expression. Moreover, any injective homomorphism of $T[U_1]+T[U_2]+\dots+T[U_k]$ to $K^{\binom{v_{e(G)}}{2}}_n$ is an injective homomorphism of $G$ to $K^{\binom{v_{e(G)}}{2}}_n$ as $K^{v_{e(G)}}_n$ is a complete graph. This shows $h(H,G)\le (1+o(1))h(H,K^{\binom{v_{e(G)}}{2}}_n)\le (1+o(1))h(H,K^{e(G)}_n)$.
\end{proof}

\begin{proof}[Proof of Theorem \ref{parosasy}]
Let $G$ be a bipartite graph on $n$ vertices with $e$ edges and let $B$ be a bipartite graph the vertices of which are covered by a star $S_k$ and a matching $\cM$ of $m\ge 0$ edges. Then 
\[
h(B,G)\le h(S_k+\cM,G)\le h(S_k,G)\cdot h(\cM,G)\le h(S_k,B^e_n)\cdot \binom{e}{m},
\]
where the first inequality follows from the fact that $G$ is bipartite and the last inequality follows from Theorem \ref{paroscsillag}. Lemma \ref{count} (ii) states that $h(S_k+\cM,B^e_n)=(1-o(1))h(S_k,B^e_n)\cdot \binom{e}{m}$. Finally observe that in $B^e_n$ all copies of $S_k+\cM$ can be extended to a copy of $B$ except those that contain the only vertex $v$ that is not adjacent to all vertices on the other side of $B$. As the number of such copies is negligible compared to the number of all copies (this is implicitly contained in the proof of Lemma \ref{count} (ii)), we obtain $h(B,G)\le h(S_k+\cM,G)=(1+o(1))h(B,B^e_n)$.
\end{proof}

% \section{Possible questions}

% $\bullet$ $ex_{bip}(n,e,K_{2,t})$

% $\bullet$ Berge type things

% $\bullet$ $ex(n,e,P_l,P_k)$ let's try to prove generalized version of the results of Ervin et al.

% $\bullet$ similar results for posets? (Alon+Sudakov)

\section{Concluding remarks}

A graph is bipartite if we forbid every odd cycle. Instead, one could forbid only the triangle. This motivates us to consider the maximum number of copies of $H$ in an $F$-free graph with $n$ vertices and $e$ edges. Let
$ex(n,e,H,F)$ denote this quantity. Note that without restricting the number of edges, this topic has attracted a lot of researchers recently, see e.g. \cite{as,ggmv,gp,gs}.

Let us note that there are some sporadic results regarding $ex(n,e,H,F)$. Fisher and Ryan \cite{fr} showed $ex(n,e,K_r,K_k)\le \binom{k-1}{r}(2e/(r-2)(r-1))^{r/2}$. Gerbner, Methuku and Vizer \cite{gmv} gave general bounds on $ex(n,e,K_r,F)$. Here we consider one particular problem and obtain an asymptotic result.

\begin{theorem}\label{k3}
$ex(n,e,P_l,K_3)=(1+o(1))N(P_l,B^e_n)$.
\end{theorem}

\begin{proof} The proof goes similarly to that of Theorems \ref{kvaziklikk} and \ref{parosasy}; first we prove the statement for a spanning forest. Here the spanning forest $F$ is a matching of size $l/2$ in case $l$ is even. In case $l$ is odd, the spanning forest $F'$ is a matching of size $(l-1)/2$ and an isolated vertex. Obviously we can choose the edges of $F$ or $F'$ at most $\binom{e}{\lfloor l/2\rfloor}$ ways, and afterwards (in case $l$ is odd), an isolated vertex can be picked $n-l+1$ ways. On the other hand in $B^e_n$ by Lemma \ref{count} (ii) there are $(1-o(1))\binom{e}{\lfloor l/2\rfloor}$ copies of $F$ and $(1-o(1))n\binom{e}{\lfloor l/2\rfloor}$ copies of $F'$, which shows $ex(n,e,F,K_3)=(1+o(1))N(F,B^e_n)$ and $ex(n,e,F',K_3)=(1+o(1))N(F',B^e_n)$.

\begin{claim} We have $ex(n,e,P_l,K_3)\le \lfloor l/2\rfloor! ex(n,e,F,K_3)$ if $l$ is even, and $ex(n,e,P_l,K_3)\le \lfloor l/2\rfloor! ex(n,e,F',K_3)/2$ if $l$ is odd.
\end{claim}
\begin{proof}
Let us consider a $K_3$-free graph $G$. Let us assume first that $l$ is even.
To count the copies of $P_l$ in $G$, first we pick a copy of $F$. %or $F'$, depending on the parity of $l$. 
Then we choose an ordering of the edges of $F$ in $\lfloor l/2 \rfloor!$ ways. Then we need to add edges connecting the $i$'th edge $u_iv_i$ to the $i+1$st edge $u_{i+1}v_{i+1}$ for every $i<\lfloor l/2\rfloor$,
% and in case $l$ is odd, we need to connect the isolated vertex $v$ to the last edge, 
in order to obtain a copy of $P_l$. It is easy to see that we count every copy of $P_l$ exactly twice this way. Thus we need to show that in there are at most two ways to complete $F$ to $P_l$ in $G$. 

Observe that if $u_1v_1$ and $u_{2}v_{2}$ are edges, then there are at most two edges that connect $u_1$ or $v_1$ to $u_{2}$ or $v_{2}$. This means there are at most two ways to choose such an edge. After that, we know which endpoint of $u_2v_2$ (say, $u_2$) we need to connect to $u_3v_3$, and only one of $u_2u_3$ and $u_2v_3$ can be in $G$. Thus there is at most one way to choose the next connection. Similarly, for any $i$ we know that the path arrives to one of $u_i$ and $v_i$ (say $u_i$) continues with the edge $u_iv_i$ and from $v_i$ it can only go to its at most one neighbor among $u_{i+1}$ and $v_{i+1}$. 
%In case $l$ is odd and $i=\lfloor l/2\rfloor$, then obviously there is only one edge from $v_i$ to $v$. 
It means there are at most two ways to complete the path after choosing the ordering of the edges of $F$, finishing the proof in case $l$ is even.

If $l$ is odd, we proceed similarly, first we pick a copy of $F'$. Then we choose an ordering of the edges of $F'$ in $\lfloor l/2 \rfloor!$ ways. Then we need co connect the edges similarly. The only difference is that now we start with $v$, it has at most one neighbor among $u_1$ and $v_1$, and we continue the path like in the previous case. Thus there is at most one way to complete the path after choosing the ordering of the edges of $F'$, finishing the proof in case $l$ is odd.

\end{proof}
Having the above claim in hand, to finish the proof it is enough to show for $l$ even that there are $(1+o(1))\lfloor l/2 \rfloor!N(F,B_n^e)$ copies of $P_l$ in $B^e_n$. The proof follows that of the above claim; observe that in a complete bipartite graph there are exactly two edges that connect $u_1$ or $v_1$ to $u_{2}$ or $v_{2}$ and exactly one edge from $v_i$ to $u_{i+1}$ or $v_{i+1}$. Thus there is equality everywhere in the above claim. Thus we are done in case $B_n^e$ is complete bipartite. Otherwise there can be some copies of $F$ such that there are less ways to extend them to copies of $P_l$. However, must be among those copies of $F$ that contain the only vertex of $B_n^e$ that is not connected to every vertex on the other side. It is easy to see that the number of these copies of $F$ is $o(N(F,B^e_n))$, thus there are still $(1+o(1))\lfloor l/2 \rfloor!N(F,B_n^e)$ copies of $P_l$ in $B^e_n$, finishing the proof in case $l$ is even.

In case $l$ is odd, the same argument finishes the proof.
%In there are exactly two such edges
%This is exactly the number of edges that connect them in $B^e_n$, finishing the proof if $l$ is even. Similarly, for any vertex and edge in a $K_3$-free graph, there is at most one edge that connects the vertex to an endpoint of the edge. Again, this is exactly the number of edges that connect them in $B^e_n$, finishing the proof.
\end{proof}

Another possible generalization is to fix the number of copies (or an upper bound on the number of copies) of a subgraph $F$ and count the copies of $H$. Fisher and Ryan \cite{fr} proved such results in case both $F$ and $H$ are cliques.

\bigskip

\textbf{Funding}: Research supported by the National Research, Development and Innovation Office - NKFIH under the grants SNN 129364 and K 116769, by the J\'anos Bolyai Research Fellowship of the Hungarian Academy of Sciences and the Taiwanese-Hungarian Mobility Program of the Hungarian Academy of Sciences.

\end{document}